\documentclass[english, reqno]{amsart}
\usepackage{babel}
\usepackage{amsmath,amssymb}
\usepackage[latin1]{inputenc}
\usepackage{hyperref}

\def\1{\raisebox{2pt}{\rm{$\chi$}}}

\newcommand{\R}{\mathbb{R}}
\newcommand{\N}{\mathbb{N}}
\newcommand{\abs}[1]{\left| #1 \right|}
\newcommand{\ol}{\overline}
\newcommand{\Om}{\Omega}
\newcommand{\dist}{\operatorname{dist}}

\newcommand{\M}{\mathcal{M}}
\newcommand{\Tr}{\mathrm{Tr}}
\renewcommand{\L}{\mathcal{L}}
\newcommand{\Rn}{\mathbb{R}^N}
\newcommand{\eps}{\varepsilon}
\renewcommand{\S}{\mathcal{S}}
\newcommand{\norm}[1]{\left\|#1\right\|} % Norm
\newcommand{\prodin}[2]{\langle #1,#2 \rangle} % Scalar product (two arguments)

\theoremstyle{plain}
\newtheorem{definition}{Definition}[section]
\newtheorem{proposition}[definition]{Proposition}
\newtheorem{theorem}[definition]{Theorem}
\newtheorem{corollary}[definition]{Corollary}
\newtheorem{lemma}[definition]{Lemma}
\newtheorem{remark}[definition]{Remark}
\theoremstyle{definition}
\theoremstyle{remark}
\numberwithin{equation}{section}

\def\vint_#1{\mathchoice%
          {\mathop{\kern 0.2em\vrule width 0.6em height 0.69678ex depth -0.58065ex
                  \kern -0.8em \intop}\nolimits_{\kern -0.4em#1}}%
          {\mathop{\kern 0.1em\vrule width 0.5em height 0.69678ex depth -0.60387ex
                  \kern -0.6em \intop}\nolimits_{#1}}%
          {\mathop{\kern 0.1em\vrule width 0.5em height 0.69678ex
              depth -0.60387ex
                  \kern -0.6em \intop}\nolimits_{#1}}%
          {\mathop{\kern 0.1em\vrule width 0.5em height 0.69678ex depth -0.60387ex
                  \kern -0.6em \intop}\nolimits_{#1}}}
                  
                  \newcommand{\aveint}[2]{\mathchoice%
          {\mathop{\kern 0.2em\vrule width 0.6em height 0.69678ex depth -0.58065ex
                  \kern -0.8em \intop}\nolimits_{\kern -0.45em#1}^{#2}}%
          {\mathop{\kern 0.1em\vrule width 0.5em height 0.69678ex depth -0.60387ex
                  \kern -0.6em \intop}\nolimits_{#1}^{#2}}%
          {\mathop{\kern 0.1em\vrule width 0.5em height 0.69678ex depth -0.60387ex
                  \kern -0.6em \intop}\nolimits_{#1}^{#2}}%
          {\mathop{\kern 0.1em\vrule width 0.5em height 0.69678ex depth -0.60387ex
                  \kern -0.6em \intop}\nolimits_{#1}^{#2}}}

\begin{document}

\title[Tug-of-war and Krylov-Safonov]
{H\"older estimate for a tug-of-war game with $1<p<2$ from Krylov-Safonov regularity theory}

\author[Arroyo]{\'Angel Arroyo}
\address{MOMAT Research Group, Interdisciplinary Mathematics Institute, Department of Applied Mathematics and Mathematical Analysis, Universidad Complutense de Madrid, 28040 Madrid, Spain}
\email{ar.arroyo@ucm.es}

\author[Parviainen]{Mikko Parviainen}
\address{Department of Mathematics and Statistics, University of Jyv\"askyl\"a, PO~Box~35, FI-40014 Jyv\"askyl\"a, Finland}
\email{mikko.j.parviainen@jyu.fi}

\date{\today}


\maketitle

\begin{abstract}
We propose a new version of the tug-of-war game and a corresponding dynamic programming principle related to the $p$-Laplacian with $1<p<2$. For this version, the asymptotic H\"older continuity of solutions can be directly derived  from recent Krylov-Safonov type regularity results in the singular case. Moreover, existence of a measurable solution can be obtained without using boundary corrections. We also establish a comparison principle.
\end{abstract}

%%%%%%%%%%%%%%%%%%%%%%%%%%%

\section{Introduction}

In Section 7 of \cite{arroyobp} we show that a solution to the dynamic programming principle
\begin{equation}\label{usual-DPP}
	u(x)
	=
	\frac{\alpha}{2}\bigg(\sup_{h\in B_1}u(x+\eps h)+\inf_{h\in B_1}u(x+\eps h)\bigg)
	+
	\beta\vint_{B_1}u(x+\eps h)\,dh
	+
	\eps^2f(x)
\end{equation}
with $\beta=1-\alpha\in(0,1]$ and $x\in\Omega\subset\R^N$ satisfies certain extremal inequalities, and thus has asymptotic H\"older regularity directly by the Krylov-Safonov theory developed in that paper. This dynamic programming principle corresponds to a version of the tug-of-war game with noise as explained in \cite{manfredipr12}, and it is linked to the $p$-Laplacian when $p\geq 2$ with suitable choice of probabilities $\alpha=\alpha(N,p)$ and $\beta=\beta(N,p)$. Interested readers can consult the references \cite{peress08}, \cite{blancr19}, \cite{lewicka20} and \cite{parviainenb} for more information about the tug-of-war games.

In the case $1<p<2$, one usually considers a variant of the game known as the tug-of-war game with orthogonal noise as in \cite{peress08}, although there is  also other recent variant covering this range and not using orthogonal noise  but measures absolutely continuous with respect to the $N$-dimensional Lebesgue measure \cite{lewicka21}. An inconvenience of the ``orthogonal noise'' approach of \cite{peress08} is that the uniform part of the measure is supported in $(N-1)$-dimensional balls. Thus we do not expect that solutions to the corresponding dynamic programming principle satisfy the extremal inequalities required for the Krylov-Safonov type regularity estimates obtained in \cite{arroyobp,arroyobp2}. 

Second, there are some deep measurability issues related to the corresponding dynamic programming principles, and this introduces some difficulties in the existence and measurability proofs in the case $1<p<2$, as explained at the beginning of Section~\ref{sec:exist}. As a matter of fact, in  \cite{hartikainen16} and \cite{arroyohp17} a modification near the boundary was necessary in order to guarantee the measurability.

 For these reasons, and with the purpose of covering the case $1<p<2$, we propose a different variant of the tug-of-war game that can be described by a dynamic programming principle having a uniform part in a $N$-dimensional ball in (\ref{eq:dpp}) below. For this variant, no boundary modifications are needed in the existence proof (Theorem~\ref{thm:existence}) because of better continuity properties that are addressed in Section~\ref{sec:existence}.  We also establish a comparison principle and thus uniqueness of solutions (Theorem~\ref{thm:comparison}). Moreover, the solutions to this dynamic programming principle are asymptotically H\"older continuous (Corollary~\ref{cor:holder}) and converge, passing to a subsequence if necessary, to a solution of the normalized $p$-Laplace equation (Theorem~\ref{thm:convergence}).

\section{Preliminaries}

We denote by $B_1$ the open unit ball of $\R^N$ centered at the origin. For $|z|=1$ we introduce the following notation,
\begin{equation}\label{operator-I}
	\mathcal{I}^z_\eps u(x)
	:=
	\frac{1}{\gamma_{N,p}}\vint_{B_1}u(x+\eps h)(z\cdot h)_+^{p-2}\,dh,
\end{equation}
for each Borel measurable bounded function $u$, where $\gamma_{N,p}$ is the normalization constant
\begin{equation}\label{gamma-ctt}
	\gamma_{N,p}
	:=
	\vint_{B_1}(z\cdot h)_+^{p-2}\,dh
	=
	\frac{1}{2}\vint_{B_1}|h_1|^{p-2}\,dh,
	\end{equation}
which is independent of the choice of $|z|=1$. Here, we have used the following notation
\begin{equation}
\label{eq:plusfunction}
	(t)_+^{p-2}
	=
	\begin{cases}
		t^{p-2} & \text{ if } t>0,\\
		0 & \text{ if } t\leq 0.
	\end{cases}
\end{equation}
We remark that $t\mapsto(t)^{p-2}_+$ is a continuous function in $\R$ when $p>2$, but it presents a discontinuity at $t=0$ when $1<p\leq 2$.
We  observe integrating  for example over a cube containing $B_1$ that $\gamma_{N,p}<\infty$ for any $p>1$. Later we compute the precise value of $\gamma_{N,p}$ in (\ref{eq:precise-constant}) but we immediately observe that $\gamma_{N,p}>\frac{1}{2}$ if  $1<p<2$ since
\begin{equation*}
	2\gamma_{N,p}
	=
	\vint_{B_1}|z\cdot h|^{p-2}\,dh
	\geq
	\vint_{B_1}|h|^{p-2}\,dh
	>
	1.
\end{equation*}

Throughout the paper $\Omega\subset\R^N$ denotes a bounded domain. For $\eps>0$, we define the $\eps$-neighborhood of $\Omega$ as
\begin{align*}
\Omega_\eps=\{x\in\R^N\,:\,\dist(x,\Omega)<\eps\}.
\end{align*}
Let $f:\Omega\to\R$ be a Borel measurable bounded function. We consider a dynamic programming principle (DPP)
\begin{align}
\label{eq:dpp}
	u(x)
	=
	\frac{1}{2}\Big(\sup_{|z|=1}\mathcal{I}^z_\eps u(x)+\inf_{|z|=1}\mathcal{I}^z_\eps u(x)\Big)
	+
	\eps^2f(x)
\end{align}
for $x\in\Omega$, with prescribed Borel measurable bounded boundary values $g:\Omega_\eps\setminus\Omega\to\R$. The parameter $p$ above is linked to the $p$-Laplace operator as explained in Section~\ref{sec:p-lap}.

Next we recall the asymptotic H\"older continuity result derived in \cite{arroyobp} and \cite{arroyobp2}. The results there apply to a quite general class of discrete stochastic processes with bounded and measurable increments and their expectations, or equivalently functions satisfying the corresponding dynamic programming principles. Moreover, the results actually hold for functions merely satisfying inequalities given in terms of extremal operators, which we recall below. This can be compared with the H\"older result for PDEs given in terms of Pucci operators (see for example \cite{caffarellic95} and \cite{krylovs79,krylovs80,trudinger80}).

For $\Lambda\geq1$, let $\M(B_\Lambda)$, as in those papers, denote the set of symmetric unit Radon measures with support in $B_\Lambda$ and $\nu:\R^N\to \M(B_\Lambda)$ such that
\begin{equation*}
	x\mapsto\int_{B_\Lambda} u(x+h) \,d\nu_x(h)
\end{equation*}
defines a Borel measurable function for every Borel measurable $u:\R^N\to \R$. By symmetric we mean that
\begin{align*}
\nu_x(E)=\nu_x(-E)
\end{align*}
holds for every measurable set $E\subset\R^N$.

\begin{definition}[Extremal operators]
\label{def:pucci}
Let $u:\R^N\to\R$ be a Borel measurable bounded function. We define the extremal Pucci type operators
\begin{equation*}
\begin{split}
	\L_\eps^+ u(x)
	:\,=
	~&
	\frac{1}{2\eps^2}\bigg(\alpha \sup_{\nu\in \M(B_\Lambda)} \int_{B_\Lambda}\delta u(x,\eps h) \,d\nu(h)+\beta\vint_{B_1} \delta u(x,\eps h)\,dh\bigg)
	\\
	=
	~&
	\frac{1}{2\eps^2}\bigg(\alpha \sup_{h\in B_\Lambda} \delta u(x,\eps h) +\beta\vint_{B_1} \delta u(x,\eps h)\,dh\bigg)
\end{split}
\end{equation*}
and
\begin{equation*}
\begin{split}
	\L_\eps^- u(x)
	:\,=
	~&
	\frac{1}{2\eps^2}\bigg(\alpha \inf_{\nu\in \M(B_\Lambda)} \int_{B_\Lambda}\delta u(x,\eps h) \,d\nu(h)+\beta\vint_{B_1} \delta u(x,\eps h)\,dh\bigg)
	\\
	=
	~&
	\frac{1}{2\eps^2}\bigg(\alpha \inf_{h\in B_\Lambda} \delta u(x,\eps h) +\beta\vint_{B_1} \delta u(x,\eps h)\,dh\bigg),
\end{split}
\end{equation*}
where $\delta u(x,\eps h)=u(x+\eps h)+u(x-\eps h)-2u(x)$ for every $h\in B_\Lambda$.
\end{definition} 
Naturally also other domains of definition are possible instead of $\R^N$ above.
\begin{theorem}[Asymptotic H\"older, \cite{arroyobp, arroyobp2}]
\label{Holder}
There exists $\eps_0>0$ such that if $u$ satisfies $\L_\eps^+ u\ge -\rho$ and $\L_\eps^- u\le  \rho$ in $B_{R}$ for some $1-\alpha=\beta>0$ where $\eps<\eps_0R$, there exist $C,\gamma>0$  such that
\[
|u(x)-u(y)|\leq \frac{C}{R^\gamma}\left(\sup_{B_{R}}|u|+R^2\rho\right)\big(|x-y|^\gamma+\eps^\gamma\big)
\]
for every $x, y\in B_{R/2}$.
\end{theorem}

It is worth remarking that the constants $C$ and $\gamma$ are independent of $\eps$, and depend exclusively on $N$, $\Lambda\geq 1$ and $\beta=1-\alpha\in(0,1]$. Also a version of Harnack's inequality \cite[Theorem 5.5]{arroyobp2} holds if the extremal inequalities are satisfied for some $\beta>0$. For a different approach for regularity in the case of tug-of-war games, see \cite{luirops13} ($p>2$) and \cite{luirop18} ($p>1$).

\section{Existence of measurable solutions with $1<p<\infty$}
\label{sec:existence}

In this section we prove existence and uniqueness of solutions to the DPP (\ref{eq:dpp}). In addition, when $1<p<2$, such  DPP  satisfies the requirements to get the asymptotic H\"older estimate from Theorem~\ref{Holder}. The regularity result and the connection of such a  DPP  (as well as the corresponding tug-of-war game) to the $p$-Laplacian are addressed in Sections~\ref{sec:reg} and \ref{sec:p-lap}, respectively.

\begin{remark}\label{rem-average}
Observe that the operator $\mathcal{I}^z_\eps$ defined in (\ref{operator-I}) is a linear average for each $|z|=1$, in the sense that $\mathcal{I}^z_\eps$ satisfies the following:
\begin{enumerate}
\item[\textit{i)}] stability: $\displaystyle\inf_{B_\eps(x)}u\leq\mathcal{I}^z_\eps u(x)\leq\sup_{B_\eps(x)}u$;
\item[\textit{ii)}] monotonicity: $\mathcal{I}^z_\eps u\leq\mathcal{I}^z_\eps v$ for $u\leq v$;
\item[\textit{iii)}] linearity: $\mathcal{I}^z_\eps(au+bv)=a\,\mathcal{I}^z_\eps u+b\,\mathcal{I}^z_\eps v$ for  $a,b\in\R$.
\end{enumerate}
\end{remark}

\subsection{Continuity estimates for $\mathcal{I}^z_\eps u$}

Given a Borel measurable bounded function $u:\Omega_\eps\to\R$, we show that the function $\mathcal{I}^z_\eps u(x)$ is continuous with respect to $x\in\overline\Omega$ and $|z|=1$. In fact, we prove that $(x,z)\mapsto\mathcal{I}^z_\eps u(x)$ is uniformly continuous. As a consequence of this, the function
\begin{equation*}
	x
	\longmapsto
	\frac{1}{2}\Big(\sup_{|z|=1}\mathcal{I}^z_\eps u(x)+\inf_{|z|=1}\mathcal{I}^z_\eps u(x)\Big)
\end{equation*}
is continuous in $\overline\Omega$ as shown in Lemma~\ref{lem:sup-inf-cont}.

\begin{lemma}
\label{lem:continuous-wrtz}
Let $\Omega\subset\R^N$. For $u:\Omega_\eps\to\R$ a Borel measurable bounded function and $x\in\overline\Omega$, the function $$z\longmapsto\mathcal{I}^z_\eps u(x)$$ is continuous on $|z|=1$. Moreover, the family $\{z\mapsto\mathcal{I}^z_\eps u(x)\,:\,x\in\overline\Omega\}$ is equicontinuous on $|z|=1$.
\end{lemma}

\begin{proof} For $|z|=|w|=1$ we have
\begin{equation*}
	\left|\mathcal{I}^z_\eps u(x)-\mathcal{I}^w_\eps u(x)\right|
	\leq
	\frac{\|u\|_\infty}{\gamma_{N,p}}\vint_{B_1}\big|(z\cdot h)^{p-2}_+-(w\cdot h)^{p-2}_+\big|\,dh,
\end{equation*}
uniformly for every $x\in\overline\Omega$. We claim that the limit
\begin{equation}\label{claim}
	\lim_{|z-w|\to 0}\vint_{B_1}\big|(z\cdot h)^{p-2}_+-(w\cdot h)^{p-2}_+\big|\,dh
	=
	0
\end{equation}
holds for every $1<p<\infty$.  Observe that the limit in (\ref{claim}) is independent of $x$ and $u$, and thus it holds uniformly for every $x\in\overline\Omega$, then the (uniform) equicontinuity of $\mathcal{I}^z_\eps u$ in $\Omega$ follows.

\textit{i) Case $p=2$.} Since $(t)^0_+=\1_{(0,\infty)}(t)$, we have
\begin{equation*}
\begin{split}
	\vint_{B_1}\big|(z\cdot h)^0_+-(w\cdot h)^0_+\big|\,dh
	=
	\frac{|B_1\cap(\{z\cdot h>0\}\triangle\{w\cdot h>0\})|}{|B_1|}
	\leq
	C |z-w|
\end{split}
\end{equation*}
for some explicit constant $C>0$ depending only on $N$, so (\ref{claim}) follows. Here $\triangle$ stands for the symmetric difference $A\triangle B=(A\setminus B)\cup(B\setminus A)$.

\textit{ii) Case $p>2$.} We observe that the function $t\mapsto(t)^{p-2}_+$ is continuous in $\R$. In addition, given any $|z|=|w|=1$, it holds that
\begin{equation*}
	|(z\cdot h)^{p-2}_+-(w\cdot h)^{p-2}_+|
	\leq 1
\end{equation*}
for each $h\in B_1$. Then (\ref{claim}) follows by the Dominated Convergence Theorem. 

\textit{iii) Case $1<p<2$.} This case requires a bit care since obtaining an integrable upper bound needed for the Dominated Convergence Theorem is not as straightforward.
To this end, we observe the inequality 
\begin{align*}
\abs{a-b}\le (a+b)\bigg(1-\frac{\min\{a,b\}}{\max\{a,b\}}\bigg).
\end{align*}
for every $a,b>0$. Thus
\begin{equation*}
	|(z\cdot h)^{p-2}_+-(w\cdot h)^{p-2}_+|
	\leq
	\big((z\cdot h)^{p-2}_++(w\cdot h)^{p-2}_+\big)\bigg(1-\frac{\min\{(z\cdot h)^{p-2}_+,(w\cdot h)^{p-2}_+\}}{\max\{(z\cdot h)^{p-2}_+,(w\cdot h)^{p-2}_+\}}\bigg).
\end{equation*}
In that way, applying H\"older inequality with $q=\frac{1}{2}\,\frac{p-3}{p-2}$ and recalling the definition of $\gamma_{N,\frac{p+1}{2}}$ we can estimate
\begin{multline*}
	\vint_{B_1}|(z\cdot h)^{p-2}_+-(w\cdot h)^{p-2}_+|\,dh
	\\
	\leq
	2\gamma_{N,\frac{p+1}{2}}^{2\frac{p-2}{p-3}}\bigg(\vint_{B_1}\bigg(1-\frac{\min\{(z\cdot h)^{p-2}_+,(w\cdot h)^{p-2}_+\}}{\max\{(z\cdot h)^{p-2}_+,(w\cdot h)^{p-2}_+\}}\bigg)^{-\frac{p-3}{p-1}}\,dh\bigg)^{-\frac{p-1}{p-3}}.
\end{multline*} 
Observe that $\tfrac{p+1}{2}>1$ and thus $\gamma_{N,\frac{p+1}{2}}<\infty$. Now, since $t\mapsto(t)^{p-2}_+$ is continuous in $(0,+\infty)$ (and we set other values identically to 0 in (\ref{eq:plusfunction})), and so, for any given $|z|=1$ and for each $h\in B_1$ such that $z\cdot h>0$, it holds that $(w\cdot h)^{p-2}_+\to(z\cdot h)^{p-2}_+$ as $w\to z$ with $|w|=1$. Then, we see that the function in the integral on the right hand side is bounded between $0$ and $1$ and converges to $0$ as $w\to z$ for each $h\in B_1$, so the assumptions in the Dominated Convergence Theorem are fulfilled, yielding that the right hand side above converges to $0$ as $w\to z$, and thus (\ref{claim}) follows.
\end{proof}

\begin{lemma}
\label{lem:continuous}
Let $\Omega\subset\R^N$. For $|z|=1$, the function 
\begin{equation*}
	x\longmapsto \mathcal{I}^z_\eps u(x)
\end{equation*}
is continuous in $\overline\Omega$ for every Borel measurable bounded function $u:\Omega_\eps\to\R$. Moreover, $\{\mathcal{I}^z_\eps u\,:\,|z|=1\}$ is equicontinuous in $\overline\Omega$.
\end{lemma}

\begin{proof}
Let $u:\Omega_\eps\to\R$ be a bounded Borel measurable function defined in $\Omega_\eps$. Our aim is to show that
\begin{equation*}
	\lim_{x,y\in\overline\Omega \text{ s.t. } |x-y|\to 0}\left|\mathcal{I}^z_\eps u(x)-\mathcal{I}^z_\eps u(y)\right|
	=
	0.
\end{equation*}
We can write
\begin{align*}
	&
	\mathcal{I}^z_\eps u(x)
	=
	\frac{1}{\gamma_{N,p}\abs{B_1}}\int_{\R^N}u(x+\eps h)\1_{B_1}(h)(z\cdot h)^{p-2}_+\,dh
	,
	\\
	&
	\mathcal{I}^z_\eps u(y)
	=
	\frac{1}{\gamma_{N,p}\abs{B_1}}\int_{\R^N}u(x+\eps h)\1_{B_1(-\frac{x-y}{\eps})}(h)(z\cdot(h+\tfrac{x-y}{\eps}))^{p-2}_+\,dh,
\end{align*}
so the following estimate follows immediately,
\begin{multline*}
	\left|\mathcal{I}^z_\eps u(x)-\mathcal{I}^z_\eps u(y)\right|
	\\
	\leq
	\frac{\norm{u}_{\infty}}{\gamma_{N,p}\abs{B_1}}\int_{\R^N}\big|\1_{B_1}(h)(z\cdot h)^{p-2}_+-\1_{B_1(-\frac{x-y}{\eps})}(h)(z\cdot(h+\tfrac{x-y}{\eps}))^{p-2}_+\big|\,dh.
\end{multline*}
We focus on the integral above. We can assume without loss of generality that $z=e_1$, otherwise we could perform a change of variables. In addition, and for the sake of simplicity, we denote $\xi=-\frac{x-y}{\eps}$. Then the result follows from the following claim,
\begin{equation}\label{claim2}
	\lim_{\xi\to 0}\int_{\R^N}\big|\1_{B_1}(h)(h_1)^{p-2}_+-\1_{B_1(\xi)}(h)(h_1-\xi_1)^{p-2}_+\big|\,dh
	=
	0.
\end{equation}
To see this we need to distinguish two cases depending on the value of $p$.

\textit{i) Case $p\geq2$.} The integrand in (\ref{claim2}) converges to zero as $\xi\to 0$ for almost every $h\in\R^N$. Moreover, it is bounded by $2$ and zero outside a bounded set. Then the claim follows by the Dominated Convergence Theorem as $\xi\to 0$.

\textit{ii) Case $1<p<2$.} In order to apply the Dominated Convergence Theorem when $1<p<2$, we observe similarly as in the proof of Lemma~\ref{lem:continuous-wrtz} that 
\begin{multline*}
	\big|\1_{B_1}(h)(h_1)^{p-2}_+-\1_{B_1(\xi)}(h)(h_1-\xi_1)^{p-2}_+\big|
	\\
\begin{split}
	\leq
	~&
	\big(\1_{B_1}(h)(h_1)^{p-2}_++\1_{B_1(\xi)}(h)(h_1-\xi_1)^{p-2}_+\big)
	\\
	~&
	\cdot\bigg(1-\frac{\min\{\1_{B_1}(h)(h_1)^{p-2}_+,\1_{B_1(\xi)}(h)(h_1-\xi_1)^{p-2}_+\}}{\max\{\1_{B_1}(h)(h_1)^{p-2}_+,\1_{B_1(\xi)}(h)(h_1-\xi_1)^{p-2}_+\}}\bigg).
\end{split}
\end{multline*}
In that way, applying H\"older inequality with $q=\frac{1}{2}\,\frac{p-3}{p-2}$,
\begin{multline*}
	\int_{D_0}\big|\1_{B_1}(h)(h_1)^{p-2}_+-\1_{B_1(\xi)}(h)(h_1-\xi_1)^{p-2}_+\big|\,dh
	\\
	\leq
	C\bigg(\int_{\R^N}\bigg(1-\frac{\min\{\1_{B_1}(h)(h_1)^{p-2}_+,\1_{B_1(\xi)}(h)(h_1-\xi_1)^{p-2}_+\}}{\max\{\1_{B_1}(h)(h_1)^{p-2}_+,\1_{B_1(\xi)}(h)(h_1-\xi_1)^{p-2}_+\}}\bigg)^{-\frac{p-3}{p-1}}\,dh\bigg)^{-\frac{p-1}{p-3}}
\end{multline*}
for every small enough $\xi$ where $C>0$ only depends on $N$ and $p$. Now again, the right hand side in the integral above is bounded between $0$ and $1$. Moreover, the integrand converges to $0$ as $\xi\to 0$ for almost every $h\in\R^N$, so that the Dominated Convergence Theorem implies (\ref{claim2}). Moreover, since the estimates obtained in this proof hold uniformly for every $x\in\overline\Omega$ and $|z|=1$, this implies the uniform equicontinuity of the family $x\mapsto\mathcal{I}^z_\eps u(x)$ with respect to $|z|=1$.
\end{proof}

As a direct consequence of the continuity estimate from the previous Lemma we get the following result.

\begin{lemma}\label{lem:sup-inf-cont}
Let $\Omega\subset\R^N$ and $u:\Omega_\eps\to\R$ be a Borel measurable bounded function. Then the function
\begin{equation*}
	x
	\longmapsto
	\frac{1}{2}\Big(\sup_{|z|=1}\mathcal{I}^z_\eps u(x)+\inf_{|z|=1}\mathcal{I}^z_\eps u(x)\Big)
\end{equation*}
is continuous in $\overline\Omega$.
\end{lemma}

\begin{proof}
The result follows directly from the equicontinuity in $\overline\Omega$ of the set of functions $\{\mathcal{I}^z_\eps u\,:\,|z|=1\}$ (Lemma~\ref{lem:continuous}) and the elementary inequalities
\begin{align*}
	\sup_{|z|=1}\mathcal{I}^z_\eps u(x)-\sup_{|z|=1}\mathcal{I}^z_\eps u(y)
	\leq
	~&
	\sup_{|z|=1}\big\{\mathcal{I}^z_\eps u(x)-\mathcal{I}^z_\eps u(y)\big\},
	\\
	\inf_{|z|=1}\mathcal{I}^z_\eps u(x)-\inf_{|z|=1}\mathcal{I}^z_\eps u(y)
	\leq
	~&
	\sup_{|z|=1}\big\{\mathcal{I}^z_\eps u(x)-\mathcal{I}^z_\eps u(y)\big\}.\qedhere
\end{align*}
\end{proof}

\subsection{Existence and uniqueness}
\label{sec:exist}

Next we show existence of Borel measurable solutions to the  DPP (\ref{eq:dpp}). We also establish a comparison principle and thus uniqueness of solutions.

We remark that, contrary to the existence proofs in \cite{hartikainen16,arroyohp17}, no boundary correction is needed as in those references, since Lemma~\ref{lem:continuous} guarantees that $u-\eps^2f$ is continuous in $\Omega$, and the solutions to (\ref{eq:dpp}) are measurable.
Also recall that measurability of operators containing $\sup$ and $\inf$ is not completely trivial as shown by Example 2.4 in \cite{luirops14}.

The proof of existence of solutions to the DPP (\ref{eq:dpp}) with prescribed values in $$\Gamma_\eps=\Omega_\eps\setminus\Omega$$ is based on Perron's method. For that, given Borel measurable bounded functions $f:\Omega\to\R$ and $g:\Gamma_\eps\to\R$, we consider the family $\S_{f,g}$ of Borel measurable functions $u:\Omega_\eps\to\R$ such that $u-\eps^2f$ is continuous in $\Omega$ and
\begin{equation}\label{sub-DPP}
	\begin{cases}
	\displaystyle
	u
	\leq
	\frac{1}{2}\Big(\sup_{|z|=1}\mathcal{I}^z_\eps u+\inf_{|z|=1}\mathcal{I}^z_\eps u\Big)+\eps^2f
	& \text{ in } \Omega,
	\\
	u
	\leq
	g
	& \text{ in }\Gamma_\eps.
	\end{cases}
\end{equation}
In the PDE literature, the corresponding class would be the class of subsolutions with suitable boundary conditions. In the following lemmas we prove that $\S_{f,g}$ is non-empty and uniformly bounded.

\begin{lemma}
\label{lem:non-empty}
Let $f$ and $g$ be Borel measurable bounded functions in $\Omega$ and $\Gamma_\eps$, respectively. There exists a Borel measurable function $u:\Omega_\eps\to\R$ such that $u-\eps^2f$ is continuous in $\Omega$ and $u$ satisfies (\ref{sub-DPP}) with $u=g$ in $\Gamma_\eps$.
\end{lemma}

\begin{proof}
Let  $C>0$ be a constant to be fixed later, fix $R=\displaystyle\sup_{x\in\Omega_\eps}|x|$ and define
\begin{equation*}
	u(x)
	=
	\begin{cases}
	C(|x|^2-R^2)+\eps^2f(x) & \text{ if } x\in\Omega,
	\\
	g(x) & \text{ if } x\in\Gamma_\eps.
	\end{cases}
\end{equation*}
Then $u-\eps^2f$ is clearly continuous in $\Omega$. To see that $u$ satisfies (\ref{sub-DPP}), let
\begin{equation*}
	u_0(x)
	=
	C(|x|^2-R^2)-\eps^2\|f\|_\infty-\|g\|_\infty
\end{equation*}
for every $x\in\Omega_\eps$. Then $u_0\leq u$ in $\Omega_\eps$. By the linearity and the monotonicity of the operator $\mathcal{I}^z_\eps$ (see Remark~\ref{rem-average}),
\begin{equation*}
\begin{split}
	\frac{1}{2}\Big(\sup_{|z|=1}\mathcal{I}^z_\eps u(x)+\inf_{|z|=1}\mathcal{I}^z_\eps u(x)\Big)
	\geq
	~&
	\frac{1}{2}\Big(\sup_{|z|=1}\mathcal{I}^z_\eps u_0(x)+\inf_{|z|=1}\mathcal{I}^z_\eps u_0(x)\Big)
	\\
	\geq
	~&
	\frac{1}{2}\inf_{|z|=1}\big\{\mathcal{I}^z_\eps u_0(x)+\mathcal{I}^{-z}_\eps u_0(x)\big\}
	\\
	=
	~&
	C\inf_{|z|=1}\bigg\{\frac{1}{2\gamma_{N,p}}\vint_{B_1}|x+\eps h|^2|z\cdot h|^{p-2}\,dh\bigg\}
	\\
	~&
	-CR^2-\eps^2\|f\|_\infty-\|g\|_\infty.
\end{split}
\end{equation*}
By the symmetry properties and the identity (\ref{integral-p3}), it turns out that
\begin{equation*}
\begin{split}
	\frac{1}{2\gamma_{N,p}}\vint_{B_1}|x+\eps h|^2|z\cdot h|^{p-2}\,dh
	=
	|x|^2+\eps^2\,\frac{N+p-2}{N+p}
	\geq
	|x|^2+\frac{\eps^2}{3}
\end{split}
\end{equation*}
holds for any $|z|=1$, $N\geq 2$ and $p>1$. Therefore
\begin{equation*}
\begin{split}
	\frac{1}{2}\Big(\sup_{|z|=1}\mathcal{I}^z_\eps u(x)+\inf_{|z|=1}\mathcal{I}^z_\eps u(x)\Big)
	\geq
	~&
	C\Big(|x|^2+\frac{\eps^2}{3}\Big)-CR^2-\eps^2\|f\|_\infty-\|g\|_\infty
	\\
	=
	~&
	u(x)-\eps^2f(x)+\Big(\frac{C\eps^2}{3}-\eps^2\|f\|_\infty-\|g\|_\infty\Big).
\end{split}
\end{equation*}
Then (\ref{sub-DPP}) follows for $C=3(\|f\|_\infty+\eps^{-2}\|g\|_\infty)$ since then the expression in the parenthesis right above equals zero.
\end{proof}

\begin{lemma}
\label{lem:boundedness}
Let $f$ be a Borel measurable bounded function in $\Omega$. If $u:\Omega_\eps\to\R$ is a Borel measurable function satisfying (\ref{sub-DPP}) in $\Omega$ then
\begin{equation*}
	\sup_\Omega u
	\leq
	C\eps^2\|f\|_\infty+\|g\|_\infty
\end{equation*}
for some constant $C>0$ depending only on $\Omega$ and $\eps$.
\end{lemma}

\begin{proof}
We start by extending the function $u$ as $\|g\|_\infty$ outside $\Omega_\eps$.
For each $x\in\Omega$ let
\begin{equation*}
	S_x
	=
	\big\{h\in B_1\,:\, \tfrac{1}{2}\leq|h|<1\ \text{ and }\ x\cdot h\geq 0\big\}.
\end{equation*}
Then we define the constant
\begin{equation*}
	\vartheta
	=
	\vartheta(N,p)
	=
	\frac{1}{2\gamma_{N,p}|B_1|}\int_{S_x}|z\cdot h|^{p-2}\,dh,
\end{equation*}
which is independent of $x\in\Omega$ and $|z|=1$. Indeed, since
\begin{equation*}
	\int_{S_x\cap\{z\cdot h<0\}}|z\cdot h|^{p-2}\,dh
	=
	\int_{S_{-x}\cap\{z\cdot h>0\}}|z\cdot h|^{p-2}\,dh
\end{equation*}
we can write
\begin{align*}
	\int_{S_x}|z\cdot h|^{p-2}\,dh&=\int_{S_x\cap\{z\cdot h>0\}}|z\cdot h|^{p-2}\,dh
	+
	\int_{S_{x}\cap\{z\cdot h<0\}}|z\cdot h|^{p-2}\,dh
\\
&
=\int_{S_x\cap\{z\cdot h>0\}}|z\cdot h|^{p-2}\,dh
	+
	\int_{S_{-x}\cap\{z\cdot h>0\}}|z\cdot h|^{p-2}\,dh.
\end{align*}
Using this and the fact that $S_x\cup S_{-x}=B_1\setminus B_{1/2}$,  we get
\begin{equation*}
	\int_{S_x}|z\cdot h|^{p-2}\,dh
	=
	\int_{B_1\setminus B_{1/2}}(z\cdot h)^{p-2}_+\,dh
	=
	\gamma_{N,p}\abs{B_1}\Big(1-\frac{1}{2^{N+p-2}}\Big).
\end{equation*}
In the last equality we used the definition of $\gamma_{N,p}$ and a change of variables as 
\begin{align*}
\int_{B_1\setminus B_{1/2}}(z\cdot h)^{p-2}_+\,dh
&=\int_{B_1}(z\cdot h)^{p-2}_+\,dh-\int_{B_{1/2}}(z\cdot h)^{p-2}_+\,dh\\
&=\int_{B_1}(z\cdot h)^{p-2}_+\,dh-2^{-(N+p-2)}\int_{B_1}(z\cdot h)^{p-2}_+\,dh.
\end{align*}
Thus we obtain
\begin{equation*}
	\vartheta
	=
	\frac{1}{2}-\frac{1}{2^{N+p-1}}
	\in(\tfrac{1}{4},\tfrac{1}{2})
\end{equation*}
for any $N\geq 2$ and $1<p<\infty$.

Let $x\in\Omega$. By Lemma~\ref{lem:continuous-wrtz}, there exists $|z_0|=1$ maximizing $\mathcal{I}^z_\eps u(x)$ among all $|z|=1$. Then
\begin{equation*}
\begin{split}
	u(x)-\eps^2f(x)
	\leq
	~&
	\frac{1}{2}\Big(\sup_{|z|=1}\mathcal{I}^z_\eps u(x)+\inf_{|z|=1}\mathcal{I}^z_\eps u(x)\Big)
	\\
	\leq
	~&
	\frac{1}{2}\big(\mathcal{I}^{z_0}_\eps u(x)+\mathcal{I}^{-z_0}_\eps u(x)\big)
	\\
	=
	~&
	\frac{1}{2\gamma_{N,p}}\vint_{B_1}u(x+\eps h)|z_0\cdot h|^{p-2}\,dh
	\\
	\leq
	~&
	\vartheta\sup_{h\in B_1\cap S_x}\{u(x+\eps h)\}
	+(1-\vartheta)\sup_{\R^N} u.
\end{split}
\end{equation*}
For each $k\in\N$, let $V_k=\R^N\setminus B_{\sqrt{k}\,\eps/2}$. Since
\begin{equation*}
	|x+\eps h|^2
	\geq
	|x|^2+\frac{\eps^2}{4}
\end{equation*}
for every $h\in B_1\cap S_x$, it turns out that $x+\eps h\in V_{k+1}$ for every $h\in B_1\cap S_x$ and $x\in V_k$. Therefore
\begin{equation*}
	\sup_{V_k}u
	\leq
	\vartheta\sup_{V_{k+1}}u
	+(1-\vartheta)\sup_{\R^N}u+\eps^2\|f\|_\infty.
\end{equation*}
Iterating this inequality starting from $V_0=\R^N$ we obtain
\begin{equation*}
	\sup_{\R^N}u
	\leq
	\vartheta^k\sup_{V_k}u
	+\bigg(\sum_{j=0}^{k-1}\vartheta^j\bigg)\big((1-\vartheta)\sup_{\R^N}u+\eps^2\|f\|_\infty\big),
\end{equation*}
and rearranging terms
\begin{equation*}
	\sup_{\R^N}u
	\leq
	\sup_{V_k}u
	+\frac{1-\vartheta^k}{\vartheta^k(1-\vartheta)}\,\eps^2\|f\|_\infty.
\end{equation*}
Since $\Omega$ is bounded, choosing large enough $k_0=k_0(\eps,\Omega)$ we ensure that $\Omega\subset B_{\sqrt{k_0}\,\eps/2}$, and thus $V_{k_0}\subset\R^N\setminus\Omega$. Thus there is necessarily a step $k\leq k_0$ so that $\displaystyle\sup_{V_k}u\leq\sup_{\R^N\setminus\Omega}u\leq\|g\|_\infty$. Using also that $\vartheta\in(\frac{1}{4},\frac{1}{2})$, we get
\begin{equation*}
	\sup_\Omega u
	\leq
	\sup_{\R^N}u
	\leq
	\|g\|_\infty
	+2^{2k+1}\,\eps^2\|f\|_\infty.
	\qedhere
\end{equation*}
\end{proof}

Now we have the necessary lemmas to work out the existence through Perron's method. The idea is to take the pointwise supremum of functions in $\S_{f,g}$, the family of Borel measurable functions $u$ with $u-\eps^2f\in C(\Omega)$ satisfying (\ref{sub-DPP}) (this would be subsolutions in corresponding PDE context), and to show that this is the desired solution. Here we also utilize the continuity of $u-\eps^2f$ in $\Om$ so that the supremum of functions in the uncountable set $\S_{f,g}$ is measurable. Indeed, otherwise to the best of our knowledge, Perron's method does not work as such  (unless $p=\infty$ \cite{lius15}) but one needs to construct a countable sequence of functions as in \cite{luirops14} to guarantee the measurability. 

\begin{theorem}
\label{thm:existence}
Let  $f$ and $g$ be Borel measurable bounded functions in $\Omega$ and $\Gamma_\eps$, respectively. There exists $u:\Omega_\eps\to\R$ a Borel measurable solution satisfying
\begin{equation}\label{DPP-g}
\begin{cases}
	\displaystyle
	u
	=
	\frac{1}{2}\Big(\sup_{|z|=1}\mathcal{I}^z_\eps u+\inf_{|z|=1}\mathcal{I}^z_\eps u\Big)+\eps^2f & \text{ in } \Omega,
	\\[1em]
	\displaystyle
	u
	=
	g & \text{ in } \Gamma_\eps.
\end{cases}
\end{equation}
\end{theorem}

\begin{proof}
In view of Lemmas~\ref{lem:non-empty} and \ref{lem:boundedness}, the set $\S_{f,g}$ is non-empty and uniformly bounded. Thus, we can define $\overline u$ as the pointwise supremum of functions in $\S_{f,g}$, that is,
\begin{equation*}
	\overline u (x)
	=
	\sup_{u\in\S_{f,g}}u(x)
\end{equation*}
for each $x\in\Omega_\eps$. The boundedness of $\overline u$ is immediate. Moreover, $\overline u$ is Borel measurable. Indeed, since $\overline u-\eps^2f$ can be expressed as the pointwise supremum of continuous functions $u-\eps^2f$ with $u\in\S_{f,g}$, it turns out that $\overline u-\eps^2f$ is lower semicontinuous in $\Omega$, and thus measurable, so the measurability of $\overline u$ follows.

By Lemma~\ref{lem:non-empty}, there exists at least one function $u$ in $\S_{f,g}$ such that $u=g$ in $\Gamma_\eps$, so $\overline u$ agrees with $g$ in $\Gamma_\eps$. On the other hand, since $\overline u\geq u$ for every $u\in\S_{f,g}$, then
\begin{equation*}
	u-\eps^2f
	\leq
	\frac{1}{2}\Big(\sup_{|z|=1}\mathcal{I}^z_\eps u+\inf_{|z|=1}\mathcal{I}^z_\eps u\Big)
	\leq
	\frac{1}{2}\Big(\sup_{|z|=1}\mathcal{I}^z_\eps\overline u+\inf_{|z|=1}\mathcal{I}^z_\eps\overline u\Big)
\end{equation*}
in $\Omega$. Taking the pointwise supremum in $\S_{f,g}$ we have that
\begin{align}
\label{eq:ol-u-sub}
\overline u-\eps^2f
	\leq
	\frac{1}{2}\Big(\sup_{|z|=1}\mathcal{I}^z_\eps\overline u+\inf_{|z|=1}\mathcal{I}^z_\eps\overline u\Big).
\end{align}
Hence $\overline u$ is a Borel measurable bounded subsolution to (\ref{sub-DPP}) with $\overline u=g$ in $\Gamma_\eps$. Next we show that $\overline u-\eps^2f$ is indeed continuous in $\Omega$. For this, let $\widetilde u:\Omega_\eps\to\R$ be the Borel measurable function defined by
\begin{equation*}
	\widetilde u
	=
	\begin{cases}
	\displaystyle\frac{1}{2}\Big(\sup_{|z|=1}\mathcal{I}^z_\eps\overline u+\inf_{|z|=1}\mathcal{I}^z_\eps\overline u\Big)+\eps^2f
	& \text{ in } \Omega,
	\\
	g & \text{ in } \Gamma_\eps.
	\end{cases}
\end{equation*}
Then $\overline u\leq\widetilde u$ in $\Omega$ by (\ref{eq:ol-u-sub}), so $\widetilde u$ is a subsolution to (\ref{sub-DPP}) since the right hand side above can be estimated from above by $\displaystyle\frac{1}{2}\Big(\sup_{|z|=1}\mathcal{I}^z_\eps\widetilde u+\inf_{|z|=1}\mathcal{I}^z_\eps\widetilde u\Big)+\eps^2f$. Observe also that $\widetilde u-\eps^2f$ is continuous in $\Omega$ by Lemma~\ref{lem:continuous}, so $\widetilde u\in\S_{f,g}$. Thus $\widetilde u\leq\overline u$, and in consequence $\overline u=\widetilde u\in\S_{f,g}$. Moreover 
\begin{equation*}
	\overline u
	=
	\frac{1}{2}\Big(\sup_{|z|=1}\mathcal{I}^z_\eps\overline u+\inf_{|z|=1}\mathcal{I}^z_\eps\overline u\Big)+\eps^2f
\end{equation*}
in $\Omega$ and the proof is finished.
\end{proof}

The uniqueness of solutions to (\ref{DPP-g}) is directly deduced from the following comparison principle. 

\begin{theorem}
\label{thm:comparison}
Let $f$ be a Borel measurable bounded function in $\Omega$, and let $u,v:\Omega_\eps\to\R$ be two Borel measurable solutions to the DPP (\ref{eq:dpp}) in $\Omega$ such that $u\leq v$ in $\Gamma_\eps$. Then $u\leq v$ in $\Omega$.
\end{theorem}

\begin{proof}
For simplicity, we define $w=u-v$. Then $w$ is continuous in $\Omega$ by Lemma~\ref{lem:sup-inf-cont} and $w\leq 0$ in $\Gamma_\eps$. Furthermore, $w$ is uniformly continuous in $\Omega$, and thus we can define $\widetilde w:\Omega_\eps\to\R$ by
\begin{equation*}
	\widetilde w(x)
	=
	\begin{cases}
	\displaystyle\lim_{\Omega\ni y\to x}w(y) & \text{ if } x\in\partial\Omega,
	\\
	w(x) & \text{ otherwise,}
	\end{cases}
\end{equation*}
so that $\widetilde w\in C(\overline\Omega)$.
Let us suppose thriving for a  contradiction that
\begin{equation*}
	M
	=
	\sup_{\Omega_\eps}w
	=
	\max_{\overline\Omega}\widetilde w
	>
	0,
\end{equation*}
where the fact that $w\leq 0$ in $\Gamma_\eps$ is used above. By continuity, the set $A=\{x\in\overline\Omega\,:\,\widetilde w(x)=M\}$ is non-empty and closed. Indeed, since $\overline\Omega$ is bounded, then $A$ is compact.
For any fixed $y\in\Omega$, by Lemma~\ref{lem:continuous-wrtz} there exist $|z_1|=|z_2|=1$ such that
\begin{equation*}
	\mathcal{I}^{z_1}_\eps u(y)=\sup_{|z|=1}\mathcal{I}^z_\eps u(y)
	\qquad\text{ and }\qquad
	\mathcal{I}^{z_2}_\eps v(y)=\inf_{|z|=1}\mathcal{I}^z_\eps v(y).
\end{equation*}
Then
\begin{equation*}
\begin{split}
	w(y)
	=
	~&
	u(y)-v(y)
	\\
	=
	~&
	\frac{1}{2}\Big(\sup_{|z|=1}\mathcal{I}^z_\eps u(y)+\inf_{|z|=1}\mathcal{I}^z_\eps u(y)\Big)
	-
	\frac{1}{2}\Big(\sup_{|z|=1}\mathcal{I}^z_\eps v(y)+\inf_{|z|=1}\mathcal{I}^z_\eps v(y)\Big)
	\\
	\leq
	~&
	\frac{1}{2}\Big(\mathcal{I}^{z_1}_\eps u(y)+\mathcal{I}^{z_2}_\eps u(y)\Big)
	-
	\frac{1}{2}\Big(\mathcal{I}^{z_1}_\eps v(y)+\mathcal{I}^{z_2}_\eps v(y)\Big)
	\\
	=
	~&
	\frac{1}{2}\Big(\mathcal{I}^{z_1}_\eps w(y)+\mathcal{I}^{z_2}_\eps w(y)\Big)
	\\
	\leq
	~&
	\sup_{|z|=1}\mathcal{I}^z_\eps w(y).
\end{split}
\end{equation*}
Using this, for any $x\in A$,
\begin{equation*}
	M
	=
	\widetilde w(x)
	=
	\lim_{\Omega\ni y\to x}w(y)
	\leq
	\lim_{\Omega\ni y\to x}\Big(\sup_{|z|=1}\mathcal{I}^z_\eps w(y)\Big)
	=
	\sup_{|z|=1}\mathcal{I}^z_\eps w(x)
	\leq
	M,
\end{equation*}
where the continuity of $x \mapsto \displaystyle\sup_{|z|=1}\mathcal{I}^z_\eps w(x)$ by Lemma~\ref{lem:continuous} has been used in the last equality. That is, $\displaystyle\sup_{|z|=1}\mathcal{I}^z_\eps w(x)=M$, and again by Lemma~\ref{lem:continuous-wrtz}, there exists $|z_0|=1$ such that
\begin{equation*}
	\frac{1}{\gamma_{N,p}}\vint_{B_1}w(x+\eps h)(z_0\cdot h)^{p-2}_+\,dh
	=
	M.
\end{equation*}
By the definition of $\gamma_{N,p}$ and the fact that $w\leq M$, this implies that $w(x+\eps h)=M$ for a.e. $|h|<1$ such that $z_0\cdot h>0$. Then $x+\eps h\in A\subset\overline\Omega$ for a.e. $|h|<1$ such that $z_0\cdot h>0$. Moreover, by the continuity of $\widetilde w$ in $\ol \Om$, it turns out that $x+\eps h\in A$ for every $|h|\leq 1$ with $z_0\cdot h\geq 0$. In particular, picking any $|h|=1$ such that $z_0\cdot h=0$ we have that $x\pm\eps h\in A$, so $x=\frac{1}{2}(x+\eps h)+\frac{1}{2}(x-\eps h)$. That is, any point $x\in A$ is the midpoint between two different points $x_1,x_2\in A$. The contradiction then follows by choosing $x\in A$ to be an extremal point of $A$, that is, a point which cannot be written as a convex combination $\lambda x_1+(1-\lambda)x_2$ of points $x_1,x_2\in A$ with $\lambda\in(0,1)$ (take for instance any point $x\in A$ maximizing the Euclidean norm among all points in $A$). Then $M\leq 0$ and the proof is finished.
\end{proof}

\section{Regularity for the tug-of-war game with $1<p<2$}
\label{sec:reg}

The above DPP can also be stochastically interpreted. It is related to a two-player zero-sum game played in a bounded domain $\Om\subset \R^N$. When the players are at $x\in \Om$, they toss a fair coin and the winner of the toss may choose $z\in \Rn,\ \abs{z}=1$, so the next point is chosen according to the probability measure
\begin{align*}
A\mapsto \frac{1}{\gamma_{N,p}}\frac{1}{\abs{B_\eps(x)}}\int_{A\cap B_\eps(x)}\Big(z\cdot \frac{h-x}{\eps}\Big)_+^{p-2}\,dh.
\end{align*}
Then the players play a new round starting from the current position. 
When the game exits the domain and the first point outside the domain is denoted by $x_{\tau}$, Player II pays Player I the amount given by $F(x_{\tau})$, where  $F:\Rn\setminus \Om\to \R$
is a given payoff function. 
Since Player I gets the payoff at then end, she tries (heuristically speaking) to maximize the outcome,  and since Player II has to pay it, he tries to minimize it.  
This is a variant of a so called tug-of-war game considered for example in \cite{peresssw09, peress08, manfredipr12}. 
As explained in more detail in those references, $u$ denotes the value
of the game, i.e.\ the expected payoff of the game when players are optimizing over their strategies. Then for this $u$ the DPP holds and it can be heuristically interpreted by considering one round of the game and summing up the different outcomes (either Player I or Player II wins the toss) with the corresponding probabilities.

Next we show that if $u$ is a solution to the DPP (\ref{eq:dpp}), then it satisfies the extremal inequalities (when $1<p<2$ and also $p=2$) needed in order to apply the H\"older result in Theorem \ref{Holder}. However, in the case $2<p<\infty$ the DPP (\ref{eq:dpp}) does not have any Pucci bounds, as we explain later in Remark~\ref{remark-p>2}.
\begin{proposition}
\label{prop:satisfies-extremal}
Let $1<p<2$ and $u$ be a bounded Borel measurable function satisfying
\begin{equation*}
	u(x)
	=
	\frac{1}{2}\Big(\sup_{|z|=1}\mathcal{I}^z_\eps u(x)+\inf_{|z|=1}\mathcal{I}^z_\eps u(x)\Big)
	+
	\eps^2f(x).
\end{equation*}
Then $\L_\eps^+u+f\geq0$ and $\L_\eps^-u+f\leq0$ for some $1-\alpha=\beta>0$ depending on $N$ and $p$, where $\L_\eps^+$ and $\L_\eps^-$ are the extremal operators as in Definition~\ref{def:pucci}.
\end{proposition}

\begin{proof}
Let
\begin{equation*}
	1-\alpha
	=
	\beta
	=
	\frac{1}{2\gamma_{N,p}}
	\in
	(0,1)
\end{equation*}
and, for $|z|=1$, consider the measure defined as
\begin{equation*}
	\nu(E)
	=
	\frac{1}{|B_1|}\int_{B_1\cap E}\frac{|z\cdot h|^{p-2}-1}{2\gamma_{N,p}-1}\,dh
\end{equation*}
for every Borel measurable set. Then $\nu\in\M(B_1)$. Indeed, by the definition of $\gamma_{N,p}$ and the fact that $\gamma_{N,p}>\frac{1}{2}$ for $1<p<2$, we have that $\nu$ is a positive measure such that $\nu(B_1)=1$.
Therefore,
\begin{equation*}
\begin{split}
	\alpha\int_{B_\Lambda}u(x+\eps h)\,d\nu(h)+\beta\vint_{B_1}u(x+\eps h)\,dh
	=
	~&
	\frac{1}{2}\cdot\frac{1}{\gamma_{N,p}}\vint_{B_1}u(x+\eps h)|z\cdot h|^{p-2}\,dh
	\\
	=
	~&
	\frac{1}{2}\Big(\mathcal{I}^z_\eps u(x)+\mathcal{I}^{-z}_\eps u(x)\Big),
\end{split}
\end{equation*}
so
\begin{equation*}
	\L^-_\eps u(x)
	\leq
	\frac{\mathcal{I}^z_\eps u(x)+\mathcal{I}^{-z}_\eps u(x)-2u(x)}{2\eps^2}
	\leq
	\L^+_\eps u(x)
\end{equation*}
for every $|z|=1$. Now, if $u$ is a solution to the DPP, then
\begin{equation*}
\begin{split}
	-f(x)
	=
	~&
	\frac{1}{2\eps^2}\Big(\sup_{|z|=1}\mathcal{I}^z_\eps u(x)+\inf_{|z|=1}\mathcal{I}^z_\eps u(x)-2u(x)\Big)
	\\
	\leq
	~&
	\sup_{|z|=1}\left\{\frac{\mathcal{I}^z_\eps u(x)+\mathcal{I}^{-z}_\eps u(x)-2u(x)}{2\eps^2}\right\}
	\\
	\leq
	~&
	\L^+_\eps u(x),
\end{split}
\end{equation*}
and similarly for $\L^-_\eps u(x)\leq-f(x)$.
\end{proof}

\begin{remark}\label{remark-p>2}
The extremal inequalities do not hold for $2<p<\infty$. Indeed, the map
\begin{equation*}
	E\mapsto\frac{1}{2\gamma_{N,p}|B_1|}\int_{B_1\cap E}|z\cdot h|^{p-2}\,dh
\end{equation*}
defines a probability measure in $B_1$ which is absolutely continuous with respect the Lebesgue measure, and whose density function vanishes as $h$ approaches an orthogonal direction to $z$ when $p>2$. Thus, it is not possible to decompose the measure as a convex combination of the uniform probability measure on $B_1$ and any probability measure $\nu$, which is an essential step in the proof of the H\"older estimate in \cite{arroyobp} and \cite{arroyobp2}.
\end{remark}

By Proposition \ref{prop:satisfies-extremal}, a solution $u$ to the DPP (\ref{eq:dpp}) satisfies the conditions of Theorem~\ref{Holder}. Thus it immediately follows that $u$ is asymptotically H\"older continuous.

\begin{corollary}
\label{cor:holder}
There exists $\eps_0>0$ such that if $u$ is a solution to the DPP (\ref{eq:dpp}) in $B_{R}$  where $\eps<\eps_0R$, there exist $C,\gamma>0$ (independent of $\eps$)  such that
\[
|u(x)-u(y)|\leq \frac{C}{R^\gamma}\left(\sup_{B_{R}}|u|+R^2\sup_{B_R}\abs{f}\right)\big(|x-y|^\gamma+\eps^\gamma\big)
\]
for every $x,y\in B_{R/2}$.

\end{corollary}
 
\section{A connection to the $p$-laplacian}
\label{sec:p-lap}

In this section we consider a connection of solutions to the DPP (\ref{eq:dpp}) to the viscosity solutions to
\begin{align}
\label{eq:normpl}
	\Delta_p^N u=-f,
\end{align}
where we now assume $f\in C(\ol\Om)$.
Here $\Delta_p^Nu$ stands for the normalized $p$-Laplacian which is the non-divergence form operator
\begin{equation*}
	\Delta_p^Nu
	=
	\Delta u+(p-2)\frac{\prodin{D^2u\nabla u}{\nabla u}}{|\nabla u|^2}.
\end{equation*}
In \cite[Section 7]{arroyobp} it was already pointed out that in the case $2 \le p<\infty$ this follows (up to a multiplicative constant) for the dynamic programming principle describing the usual tug-of-war game (\ref{usual-DPP}), so here the main interest lies in the case $1<p<2$.

First, to establish the connection to the $p$-Laplace equation, we need to derive asymptotic expansions related to the DPP (\ref{eq:dpp}) for $C^2$-functions.
The expansion below holds for the full range $1<p<\infty$.
\begin{proposition}
\label{prop:asymp-exp}
Let $u\in C^2(\Omega)$. If $1<p<\infty$, then
\begin{equation}
\label{eq:asymp}
\begin{split}
	\mathcal{I}^z_\eps u(x)
	=
	~&
	u(x)+\eps\,\frac{\gamma_{N,p+1}}{\gamma_{N,p}}\,\nabla u(x)\cdot z
	\\
	~&
	+
	\frac{\eps^2}{2(N+p)}\left[\Delta u(x)+(p-2)\prodin{D^2u(x)z}{z}\right]+o(\eps^2).
\end{split}
\end{equation}

In particular, if $\nabla u(x)\neq 0$ and $z^*=\frac{\nabla u(x)}{|\nabla u(x)|}$, then
\begin{equation*}
	\lim_{\eps\to0}\frac{\mathcal{I}^{z^*}_\eps u(x)+\mathcal{I}^{-z^*}_\eps u(x)-2u(x)}{2\eps^2}
	=
	\frac{1}{2(N+p)}\,\Delta^N_pu(x).
\end{equation*}
\end{proposition}

\begin{proof}
For the sake of simplicity, we use the notation for the tensor product of (column) vectors in $\R^N$, $v\otimes w=vw^\top$, which allows to write $\prodin{Mv}{v}=\Tr\left\{M\,v\otimes v\right\}$. Using the second order Taylor's expansion of $u$ we obtain
\begin{align}
\label{eq:expansion}
	\frac{\mathcal{I}^z_\eps u(x)-u(x)}{\eps}
	=
	~&
	\frac{1}{\gamma_{N,p}}\vint_{B_1}\left(\nabla u(x)\cdot h+\frac{\eps}{2}\,\Tr\left\{D^2u(x) h\otimes h\right\}+o(\eps)\right)(z\cdot h)_+^{p-2}\,dh\nonumber
	\\
	=
	~&
	\nabla u(x)\cdot\left(\frac{1}{\gamma_{N,p}}\vint_{B_1}h\,(z\cdot h)_+^{p-2}\,dh\right)\nonumber
	\\
	~&
	+
	\frac{\eps}{2}\,\Tr\left\{D^2u(x)\left(\frac{1}{\gamma_{N,p}}\vint_{B_1}h\otimes h\,(z\cdot h)_+^{p-2}\,dh\right)\right\}+o(\eps).
\end{align}

In order to compute the first integral in the right-hand side of the previous identity, let $R$ be any orthogonal transformation such that $Re_1=z$ i.e.\ $e_1=R^\top z$. Then a change of variables $Rw=h$ yields
	\begin{align*}
\vint_{B_1}h\,(z\cdot h)_+^{p-2}\,dh
	=
	R\vint_{B_1}w\,(w_1)_+^{p-2}\,dw
\end{align*}
using $z\cdot Rw=z^\top Rw=(R^\top z)^\top w=e_1^\top w=w_1$. Going back to the original notation and observing by symmetry that
\begin{equation*}
	\vint_{B_1}h_i(h_1)_+^{p-2}\,dh
	=
	\begin{cases}
	\displaystyle\vint_{B_1}(h_1)_+^{p-1}\,dh
	=
	\gamma_{N,p+1}
	&
	\text{ if } i=1,
	\\
	0
	&
	\text{ if } i\neq1,
	\end{cases}
\end{equation*}
we get
\begin{align*}
\vint_{B_1}h\,(z\cdot h)_+^{p-2}\,dh=R\vint_{B_1}h\,(h_1)_+^{p-2}\,dh=\gamma_{N,p+1}Re_1=\gamma_{N,p+1}z.
\end{align*}
Next we repeat the change of variables in the second integral on the right hand side of (\ref{eq:expansion}) to get 
\begin{equation*}
\begin{split}
	\frac{1}{\gamma_{N,p}}\vint_{B_1}h\otimes h\,(z\cdot h)_+^{p-2}\,dh
	=
	~&
	R\left(\frac{1}{\gamma_{N,p}}\vint_{B_1}h\otimes h\,(h_1)_+^{p-2}\,dh\right)R^\top.
\end{split}
\end{equation*}
Observe that the integral in parenthesis above is a diagonal matrix. Indeed, for $i\neq j$, by symmetry,
\begin{equation*}
	\frac{1}{\gamma_{N,p}}\vint_{B_1}h_ih_j\,(h_1)_+^{p-2}\,dh
	=
	0.
\end{equation*}
In order to compute the diagonal elements, we utilize the explicit values of the normalization constants from Lemma~\ref{integral-p}, then
\begin{equation*}
	\frac{1}{\gamma_{N,p}}\vint_{B_1}h_i^2\,(h_1)_+^{p-2}\,dh
	=
	\frac{1}{2\gamma_{N,p}}\vint_{B_1}h_i^2\,|h_1|^{p-2}\,dh
	=
	\begin{cases}
	\frac{p-1}{N+p}, & i=1,\\[5pt]
	\frac{1}{N+p}, & i=2,\ldots,n.
	\end{cases}
\end{equation*}

Combining, we get
\begin{align*}
	\frac{1}{\gamma_{N,p}}\vint_{B_1}h\otimes h\,(z\cdot h)_+^{p-2}\,dh
	=~&
	R\left(\frac{1}{N+p}\,(I-e_1\otimes e_1)+\frac{p-1}{N+p}\,e_1\otimes e_1\right)R^\top
	\\
	=
	~&
	\frac{1}{N+p}\left(I+(p-2)z\otimes z\right).
\end{align*}
The proof is concluded after replacing these integrals in the expansion for $\mathcal{I}^z_\eps u(x)$.
\end{proof}

Next we show that the solutions to the DPP (\ref{eq:dpp}) converge  uniformly as $\eps\to 0$ to a viscosity solution of 
\begin{align*}
\Delta_p^Nu=-2(N+p)f.
\end{align*}

But before, we recall the definition of viscosity solutions for the convenience. Below $\lambda_{\max} (D^2\phi(x_0))$ and $\lambda_{\min} (D^2\phi(x_0))$ refer to the largest and smallest eigenvalue, respectively, of $D^2\phi(x_0)$. This definition is equivalent to the standard way of defining viscosity solutions through convex envelopes. Different definitions of viscosity solutions in this context are analyzed for example in \cite[Section 2]{kawohlmp12}. 

\begin{definition}
\label{eq:def-normalized-visc}
Let $\Om\subset\R^N$ be a bounded domain and $1<p<\infty$. A lower semicontinuous function $u$ is  a viscosity supersolution of (\ref{eq:normpl}) if for all $x_0\in\Om$  and $\phi\in C^2(\Om)$ such that $u-\phi$ attains a local minimum at $x_0$, one has
\begin{equation*}
\begin{cases}
	\Delta_p^N \phi(x_0)\le -f(x_0)\ &\text{if}\quad \nabla\phi(x_0)\neq 0,\\
	\Delta\phi(x_0)+(p-2)\lambda_{\max} (D^2\phi(x_0))\le -f(x_0)\ &\text{if}\quad\nabla\phi(x_0)=0\,\,\text{and}\,\, p\geq 2,\\
	\Delta\phi(x_0)+(p-2)\lambda_{\min} (D^2\phi(x_0))\le -f(x_0)\ &\text{if}\quad\nabla\phi(x_0)=0\,\,\text{and}\,\, 1<p<2.
\end{cases}
\end{equation*}
An upper semicontinuous function $u$ is a viscosity subsolution of (\ref{eq:normpl}) if $-u$ is a supersolution. We say that $u$ is a viscosity solution of (\ref{eq:normpl}) in $\Om$ if it is both a viscosity sub- and supersolution.
\end{definition}

\begin{theorem}
\label{thm:convergence}
Let $1<p<2$ and $\{u_\eps\}$ be a family of uniformly bounded Borel measurable solutions to the DPP (\ref{eq:dpp}). Then there is a subsequence and a H\"older continuous function $u$ such that 
\begin{align*}
u_{\eps}\to u \quad \text{locally uniformly.}
\end{align*} 
Moreover, $u$ is a viscosity solution to $\Delta_p^Nu=-2(N+p)f$.
\end{theorem}

\begin{proof}
First  we can use Asymptotic Arzel\`a-Ascoli's Theorem \cite[Lemma 4.2]{manfredipr12} in connection to Theorem~\ref{Holder} to find a locally uniformly converging subsequence to a H\"older continuous function.
 Then it remains to verify that the limit is a viscosity solution to the $p$-Laplace equation. For $\phi\in C^2(\Omega)$, fix $x\in \Om$. By Lemma~\ref{lem:continuous-wrtz}, there exists $|z_1^\eps|=1$ such that 
\[
\mathcal{I}^{z_1^\eps}_\eps\phi(x)=\inf_{|z|=1}\mathcal{I}^z_\eps \phi(x).
\]
By Proposition~\ref{prop:asymp-exp},
\begin{multline}
\label{eq:approx-ineq}
	\frac{\inf_{|z|=1}\mathcal{I}^{z}_\eps \phi(x)+\sup_{|z|=1}\mathcal{I}^{z}_\eps \phi(x)-2\phi(x)}{2\eps^2}
	\\
	\begin{split}
	\ge
	~&
	\frac{\mathcal{I}^{z_1^\eps}_\eps \phi(x)+\mathcal{I}^{-z_1^\eps}_\eps \phi(x)-2\phi(x)}{2\eps^2}
	\\
	=
	~&
	\frac{1}{2(N+p)}\left[\Delta \phi(x)+(p-2)\Tr\left\{D^2\phi(x)\, z_1^\eps\otimes z_1^\eps\right\}\right]+\frac{o(\eps^2)}{\eps^{2}}.
	\end{split}
\end{multline}

Let $u$ be the H\"older continuous limit obtained as a uniform limit of the solutions to the DPP. Choose a point $x_0\in \Omega$ and a $C^2$-function  $\phi$ defined in a neighborhood of $x_0$ touching $u$ at $x_0$ from below.
By the uniform convergence, there exists a sequence $x_{\eps} $ converging to $x_0$ such that $u_{\eps} - \phi $ has a minimum at $x_{\eps}$ (see \cite[Section 10.1.1]{evans10}) up to an error $\eta_{\eps}>0$, that is, there exists $x_{\eps}$ such that
\begin{equation*}
	u_{\eps} (y) - \phi (y)
	\geq
	u_{\eps} (x_{\eps}) - \phi(x_{\eps})-\eta_{\eps}
\end{equation*}
at the vicinity of $x_{\eps}$. The arbitrary error $\eta_{\eps}$ is due to the fact that $u_{\eps}$ may be discontinuous and we might not attain the infimum. Moreover, by adding a constant, we may assume that $\phi(x_{\eps}) = u_{\eps} (x_{\eps})$ so that $\phi$ approximately touches $u_{\eps}$ from below. Recalling the fact that $u_\eps$ is a solution to the DPP (\ref{eq:dpp}) and that $\mathcal{I}^z_\eps$ is monotone and  linear (see Remark~\ref{rem-average}), we have that
\begin{equation*}
	\mathcal{I}^z_\eps u_\eps(x_\eps)
	\geq
	\mathcal{I}^z_\eps\phi (x_\eps)+u_\eps(x_\eps)-\phi(x_\eps)-\eta_\eps.
\end{equation*}
Thus, by choosing $\eta_{\eps}=o(\eps^2)$, we obtain
\[
 \frac{o(\eps^2)}{\eps^2} \ge \frac{\inf_{|z|=1}\mathcal{I}^{z}_\eps \phi(x_{\eps})+\sup_{|z|=1}\mathcal{I}^{z}_\eps \phi(x_{\eps})-2\phi(x_{\eps})+2\eps^2 f(x_{\eps})}{2\eps^2}.
\]
Using (\ref{eq:approx-ineq}) at $x_{\eps}$ and combining this with the previous estimate, we obtain
\begin{align}
\label{eq:final-expansion}
	-f(x_{\eps})+\frac{o(\eps^2)}{\eps^2} \ge\frac{1}{2(N+p)}\left[\Delta \phi(x_{\eps})+(p-2)\Tr\left\{D^2\phi(x_{\eps})\, z_1^\eps\otimes z_1^\eps\right\}\right].
\end{align}

Let us assume first that $\nabla\phi(x_0)\neq 0$. By (\ref{eq:asymp}), we see that 
\begin{align*}
\lim_{\eps\to 0}z_{1}^{\eps}=-\frac{\nabla \phi(x_0)}{\abs{\nabla \phi(x_0)}},
\end{align*}
and thus we end up with
\begin{align*}
	-f(x_0) \ge\frac{1}{2(N+p)}\Delta_p^N \phi(x_0).
\end{align*}

Finally we consider the case $\nabla\phi(x_0)=0$. Similarly as above, (\ref{eq:final-expansion}) follows. Even if we now have no information on the convergence of $z_1^\eps$, since
\begin{equation*}
	\lambda_{\min}(M)
	\leq
	\Tr\left\{M\,z\otimes z\right\}
	\leq
	\lambda_{\max}(M)
\end{equation*}
for every $|z|=1$, we still can deduce
\begin{align*}
\begin{cases}
\Delta\phi(x_0)+(p-2)\lambda_{\min} (D^2\phi(x_0))\le -2(N+p)f(x_0)\quad &\text{if}\quad p\geq 2,\\
\Delta\phi(x_0)+(p-2)\lambda_{\max} (D^2\phi(x_0))\le -2(N+p)f(x_0)\quad &\text{if}\quad 1<p<2.
\end{cases}
\end{align*}

Thus we have shown that $u$ is a viscosity supersolution to the $p$-Laplace equation. Similarly, starting with $\sup$ instead of $\inf$, we can show that $u$ is a subsolution, and thus a solution. 
\end{proof}

\begin{remark}
For  $1<p<\infty$,  $u\in C^2(\Omega)$ and $x\in\Omega$ such that $\nabla u(x)\neq 0$ we could also show  that
\begin{equation*}
\begin{split}
	\frac{1}{2}\bigg(\sup_{|z|=1}\mathcal{I}^z_\eps u(x)+\inf_{|z|=1}\mathcal{I}^z_\eps u(x)\bigg)
	=
	~&
	u(x)+
	\frac{\eps^2}{2(N+p)}\,\Delta^N_pu(x)+o(\eps^2).
\end{split}
\end{equation*}
Indeed, by working carefully through the estimates similarly as in \cite[Lemmas 2.1--2.2]{peress08}, we could show
\begin{equation*}
\begin{split}
	\frac{1}{2}\bigg(\sup_{|z|=1}\mathcal{I}^z_\eps u(x)+\inf_{|z|=1}\mathcal{I}^z_\eps u(x)\bigg)
	=
	~&
	\frac{\mathcal{I}^{z^*}_\eps u(x)+\mathcal{I}^{-z^*}_\eps u(x)}{2}+o(\eps^2),
\end{split}
\end{equation*}
for $z^*=\frac{\nabla u(x)}{|\nabla u(x)|}$.
By Proposition~\ref{prop:asymp-exp}, we have
\begin{align*}
\frac{\mathcal{I}^{z^*}_\eps u(x)+\mathcal{I}^{-z^*}_\eps u(x)}{2}+o(\eps^2)
	=
	~&
	u(x)+\frac{\eps^2}{2(N+p)}\,\Delta^N_pu(x)+o(\eps^2),
\end{align*}
and combining these estimates we obtain the desired estimate. This would give an alternative way to write down the proof that the limit is a $p$-harmonic function. Reading this expansion in a viscosity sense gives a different characterization of $p$-harmonic functions as in \cite{manfredipr10} and \cite{kawohlmp12} by the same proof as above.  
\end{remark}

%%%%%%%%%%%%%%%%%%%%%%%%%%%%%%

\appendix 
\section{Some useful integrals}

In the appendix, we record some useful integrals that, no doubt, are known to the experts but are hard to find in the literature.

\begin{lemma}
Let $\alpha_1,\ldots,\alpha_N>-1$. Then
\begin{equation}\label{integrals0}
	\vint_{B_1}|h_1|^{\alpha_1}\cdots|h_N|^{\alpha_N}\,dh_1\cdots dh_N
	=
	\frac{1}{\pi^{N/2}}\cdot\frac{\Gamma(\frac{N}{2}+1)\Gamma(\frac{\alpha_1+1}{2})\cdots\Gamma(\frac{\alpha_N+1}{2})}{\Gamma(\frac{N+\alpha_1+\cdots+\alpha_N+2}{2})}.
\end{equation}
\end{lemma}

\begin{proof}
For convenience, we denote by $B_1^N$ the $N$-dimensional unit ball centered at the origin. We decompose the integral over $B_1^N$ by integrating in the first place with respect to $t=h_N$, that is
\begin{multline*}
	\int_{B_1^N}|h_1|^{\alpha_1}\cdots|h_N|^{\alpha_N}\,dh_1\cdots dh_N
	\\
	=
	\int_{-1}^1|t|^{\alpha_N}\left(\int_{\sqrt{1-t^2}\,B_1^{N-1}}|h_1|^{\alpha_1}\cdots|h_{N-1}|^{\alpha_{N-1}}\,dh_1\cdots dh_{N-1}\right)\,dt.
\end{multline*}
Then the change of variables $\sqrt{1-t^2}(w_1,\ldots,w_{N-1})=(h_1,\ldots,h_{N-1})$  and returning to the original notation gives
\begin{align*}
	&\int_{B_1^N}|h_1|^{\alpha_1}\cdots|h_N|^{\alpha_N}\,dh_1\cdots dh_N
	\\
	&=
	\int_{-1}^1|t|^{\alpha_N}(1-t^2)^{\frac{N+\alpha_1+\ldots+\alpha_{N-1}-1}{2}}\,dt
	\cdot\int_{B_1^{N-1}}|h_1|^{\alpha_1}\cdots|h_{N-1}|^{\alpha_{N-1}}\,dh_1\cdots dh_{N-1}.
\end{align*}

Next we focus on the the first integral in the right hand side. Using the symmetry properties and performing a change of variables we get
\begin{align*}
\int_{-1}^1|t|^{\alpha_N}(1-t^2)^{\frac{N+\alpha_1+\ldots+\alpha_{N-1}-1}{2}}\,dt
	=
	~&
	2\int_0^1t^{\alpha_N}(1-t^2)^{\frac{N+\alpha_1+\ldots+\alpha_{N-1}-1}{2}}\,dt
	\\
	=
	~&
	\int_0^1t^{\frac{\alpha_N-1}{2}}(1-t)^{\frac{N+\alpha_1+\ldots+\alpha_{N-1}-1}{2}}\,dt
	\\
	=
	~&
	\frac{\Gamma(\tfrac{\alpha_N+1}{2})\Gamma(\tfrac{N+\alpha_1+\ldots+\alpha_{N-1}+1}{2})}{\Gamma(\tfrac{N+\alpha_1+\ldots+\alpha_{N-1}+\alpha_N+2}{2})},
\end{align*}
where for the last equality we have recalled the well-known formula arising in connection to the $\beta$-function
\begin{equation*}
	\int_0^1t^{x-1}(1-t)^{y-1}\,dt
	=
	\frac{\Gamma(x)\Gamma(y)}{\Gamma(x+y)}
\end{equation*}
for $x,y>0$. In the same way, we in general obtain
\begin{multline*}
	\int_{B_1^{N-k}}|h_1|^{\alpha_1}\cdots|h_{N-k}|^{\alpha_{N-k}}\,dh_1\cdots dh_{N-k}
	\\
\begin{split}
	=&
	\int_{-1}^1|t|^{\alpha_{N-k}}(1-t^2)^{\frac{\alpha_1+\ldots+\alpha_{N-k-1}+N-k-1}{2}}\,dt
	\\
	&\hspace{1 em}\cdot\int_{B_1^{N-k-1}}|h_1|^{\alpha_1}\cdots|h_{N-k-1}|^{\alpha_{N-k-1}}\,dh_1\cdots dh_{N-k-1}\nonumber
\\
=&\int_{0}^1|t|^{\frac{\alpha_{N-k}+1}{2}-1}(1-t)^{\frac{\alpha_1+\ldots+\alpha_{N-k-1}+N-k+1}{2}-1}\,dt
	\\
	&\hspace{1 em}\cdot\int_{B_1^{N-k-1}}|h_1|^{\alpha_1}\cdots|h_{N-k-1}|^{\alpha_{N-k-1}}\,dh_1\cdots dh_{N-k-1}
\end{split}
\end{multline*}

Iterating the above formula, and dividing out the repeating terms we get
\begin{multline*}
\int_{B_1^N}|h_1|^{\alpha_1}\cdots|h_N|^{\alpha_N}\,dh_1\cdots dh_N\\
= \frac{1}{\Gamma(\tfrac{\alpha_1+\ldots+\alpha_N+N+2}{2})}\Gamma(\tfrac{\alpha_N+1}{2})\cdots \Gamma(\tfrac{\alpha_2+1}{2}) \Gamma(\tfrac{\alpha_1+3}{2})\int_{-1}^1|h_1|^{\alpha_1}\,dh_1
\end{multline*}
Then using the definition of the Gamma function and integration by parts, we get
\begin{equation*}
	\Gamma(\tfrac{\alpha_1+3}{2})\int_{-1}^1|h_1|^{\alpha_1}\,dh_1
	=\Gamma(\tfrac{\alpha_1+1}{2}+1)\frac{2}{\alpha_1+1}
	=\Gamma(\tfrac{\alpha_1+1}{2}).
\end{equation*}
Finally, the result follows by recalling that
\begin{equation*}
	|B_1^N|
	=
	\frac{\pi^{N/2}}{\Gamma(\tfrac{N}{2}+1)}.\qedhere
\end{equation*}
\end{proof}

 The next lemma follows as a special case of the previous result. This lemma is the one we actually use in the proofs.
\begin{lemma}\label{integral-p}
Let $1<p<\infty$. Then
\begin{equation}\label{integral-p0}
	\frac{1}{2\gamma_{N,p}}\vint_{B_1}|h_1|^p\,dh
	=
	\frac{p-1}{N+p}
\end{equation}
and
\begin{equation}\label{integral-p2}
	\frac{1}{2\gamma_{N,p}}\vint_{B_1}|h_1|^{p-2}|h_2|^2\,dh
	=
	\frac{1}{N+p}.
\end{equation}
\end{lemma}

\begin{proof}
First we recall the definition of $\gamma_{N,p}$ in (\ref{gamma-ctt}) and use (\ref{integrals0}) to obtain 
\begin{equation}
\label{eq:precise-constant}
	\gamma_{N,p}
	=
	\vint_{B_1}(z\cdot h)_+^{p-2}\,dh
	=
	\frac{1}{2}\vint_{B_1}|h_1|^{p-2}\,dh
	=
	\frac{1}{2\sqrt{\pi}}\cdot\frac{\Gamma(\frac{N}{2}+1)\Gamma(\frac{p-1}{2})}{\Gamma(\frac{N+p}{2})}.
\end{equation}
\sloppy
Since $\Gamma(s+1)=s\,\Gamma(s)$ and $\Gamma(\frac{1}{2})=\sqrt\pi$, applying the identity (\ref{integrals0}) with $\alpha_1=p$ and $\alpha_2=\ldots=\alpha_N=0$ we have
\begin{equation*}
	\vint_{B_1}|h_1|^p\,dh
	=
	\frac{1}{\sqrt{\pi}}\cdot\frac{\Gamma(\frac{N}{2}+1)\Gamma(\frac{p+1}{2})}{\Gamma(\frac{N+p+2}{2})}
	=
	\frac{p-1}{N+p}\cdot\frac{1}{\sqrt{\pi}}\cdot\frac{\Gamma(\frac{N}{2}+1)\Gamma(\frac{p-1}{2})}{\Gamma(\frac{N+p}{2})},
\end{equation*}
and (\ref{integral-p0}) follows by combining the previous formulas. Similarly, since $\Gamma(\frac{3}{2})=\frac{\sqrt\pi}{2}$,
\begin{equation*}
\begin{split}
	\vint_{B_1}|h_1|^{p-2}|h_2|^2\,dh
	=
	\frac{1}{\pi}\cdot\frac{\Gamma(\frac{N}{2}+1)\Gamma(\frac{p-1}{2})\Gamma(\frac{3}{2})}{\Gamma(\frac{N+p+2}{2})}
	=
	\frac{1}{N+p}\cdot\frac{1}{\sqrt{\pi}}\cdot\frac{\Gamma(\frac{N}{2}+1)\Gamma(\frac{p-1}{2})}{\Gamma(\frac{N+p}{2})},
\end{split}
\end{equation*}
and (\ref{integral-p2}) follows.
\end{proof}

\begin{corollary}
Let $1<p<\infty$ and $|z|=1$. Then
\begin{equation}\label{integral-p3}
	\frac{1}{2\gamma_{N,p}}\vint_{B_1}|h|^2|z\cdot h|^{p-2}\,dh
	=
	\frac{N+p-2}{N+p}.
\end{equation}
\end{corollary}

\begin{proof}
By symmetry, we can take $z=e_1$, so $z\cdot h=h_1$. Then
$|h|^2|h_1|^{p-2}=|h_1|^p+|h_1|^{p-2}|h_2|^2+\ldots+|h_1|^{p-2}|h_N|^2$, so
\begin{equation*}
	\vint_{B_1}|h|^2|z\cdot h|^{p-2}\,dh
	=
	\vint_{B_1}|h_1|^p\,dh+(N-1)\vint_{B_1}|h_1|^{p-2}|h_2|\,dh,
\end{equation*}
so (\ref{integral-p3}) follows from (\ref{integral-p0}) and (\ref{integral-p2}).
\end{proof}

%%%%%%%%%%%%%%%

\noindent\textbf{Acknowledgements.}
\'A.~A.\ is supported by grant PID2021-123151NB-I00.

\end{document}